\UseRawInputEncoding
\documentclass[11pt,onecolumn, reqno]{amsart}
\usepackage[pagewise]{lineno}
\newtheorem{theorem}{Theorem}[section]
\usepackage[left=1.5cm,top=1.5cm,bottom=1.5cm,right=1.5cm]{geometry}
\usepackage{latexsym}
\usepackage{amssymb,amsmath,amsfonts}

\usepackage{multirow}
\usepackage{graphicx}
\usepackage[colorlinks = true]{hyperref}
\newtheorem{lemma}[theorem]{Lemma}
\newtheorem{corollary}[theorem]{Corollary}

\theoremstyle{definition}
\newtheorem{definition}[theorem]{Definition}

\newtheorem{example}[theorem]{Example}
\newtheorem{remark}[theorem]{Remark}

\newcommand{\R}{\mathbb R}

\newcommand{\N}{\mathbb N}

\newcommand{\Lcal}{\mathcal{L}}
\newcommand{\Mcal}{\mathcal{M}}

\numberwithin{equation}{section}
 \global\parskip 6pt \oddsidemargin
0.1cm \evensidemargin 0.4cm \headheight 0.1cm \topmargin -2.0cm
\begin{document}

\pagenumbering{arabic}

  \title[Algorithm for equations of Hammerstein type and applications]
 {Algorithm for equations of Hammerstein type and applications}

\author{M.O. Aibinu$^{1*}$, S. C. Thakur$^2$, S. Moyo$^3$}
\address{$^{1}$ Institute for Systems Science \& KZN E-Skill CoLab, Durban University of Technology, Durban 4000, South Africa}
\address{$^{1}$  DSI-NRF Centre of Excellence in Mathematical and Statistical Sciences (CoE-MaSS), South Africa}
\address{$^{1}$ National Institute for Theoretical and Computational Sciences (NITheCS), South Africa}
\address{$^{2}$ KZN E-Skill CoLab, Durban University of  Technology, Durban 4000, South Africa}
\address{$^{3}$Institute for Systems Science \& Office of the DVC Research, Innovation \& Engagement Milena Court, Durban University of Technology, Durban 4000, South Africa}
\email{$^*$moaibinu@yahoo.com / mathewa@dut.ac.za}

 \keywords{ Hammerstein equation; monotone type mapping; strong convergence; forced oscillations; finite amplitude of a pendulum.\\
{\rm 2010} {\it Mathematics Subject Classification}: 47H06; 47J25.}

\begin{abstract}
Equations of Hammerstein type cover large variety of areas and are of much interest to a wide audience due to the fact that they have applications in numerous areas. Suitable conditions are imposed to obtain a strong convergence result for nonlinear integral equations of Hammerstein type with monotone type mappings. A technique which does not involve the assumption of existence of a real constant whose calculation is unclear has been used in this study to obtain the strong convergence result. Moreover, our technique is applied to show the forced oscillations of finite amplitude of a pendulum as a specific example of nonlinear integral equations of Hammerstein type. Numerical example is given for the illustration of the convergence of the sequences of iteration. These are done to demonstrate to our readers that this approach can be applied to problems arising in physical systems.
\end{abstract}

\maketitle

\section{Introduction}

The concept of monotone operators which was introduced in the 1960s has been successfully applied by many researchers to the equation of Hammerstein type.
 A nonlinear integral equation of Hammerstein type on $\Omega$  (see, e.g., Hammerstein \cite{b10}) is one of the form
\begin{equation}\label{e2}
u(x) + \int_{\Omega}k(x,y)f(y, u(y))dy = h(x),
\end{equation}
where $dy$ stands for a $\sigma$-finite measure on the measure space $\Omega$,
the kernel $k$ is defined on $\Omega \times \Omega$, $f$ is a real-valued function defined on $\Omega \times \R$ and is in general nonlinear, $h$ is a given function on $\Omega$ and $u$ is the unknown function defined on $\Omega$. Let $X$ be a real Banach space, $X^*$ its dual, $F:X\rightarrow X^*$ a nonlinear mapping of $X$ into $X^*$ and $K: X^*\rightarrow X$ a nonlinear mapping of $X^*$ into $X.$  The abstract form of Eq. (\ref{e2}) is given by
\begin{equation}\label{e3}
u + KFu = 0,
\end{equation}
where  $u \in X$ (see, e.g, Aibinu and Chidume \cite{b39}, Chidume and Bello \cite{cu1},  Chidume and Yekini (\cite{r6}, Daman \cite{DTZ1}, Diop et al. \cite{{b40}}). Hammerstein equations cover a large variety of areas and are of much interest to a wide audience due to the fact that
they have applications in numerous areas. Several problems that arise in differential equations (ordinary and partial), for instance, elliptic boundary value problems whose linear parts possess Green's function can be transformed into the Hammerstein integral equations. Equations of the Hammerstein type play a crucial role in the theory of optimal control systems and in automation and network theory (see, e.g.,  Dolezale \cite{b14}). Assuming existence, approximating a solution of $Au = 0,$ where $A : X \rightarrow X^*$ is of monotone-type has been a recent subject of interest to the researchers (see, e.g, Aibinu and Mewomo \cite{th4}, \cite{b4}, Aibinu et al. \cite{Aibinu3}, Chidume et al. \cite{ChidumeAN}, Chidume and Bello \cite{cu2}, Chidume and Idu \cite{b28}).  Strong convergence theorems were established for nonlinear Hammerstein equation Eq. (\ref{e3}) under the assumption of existence of a real constant whose calculation is unclear (see e.g, Chidume and Bello \cite{cu1}, Chidume and Djitte \cite{r5}, Chidume and Idu \cite{b28}, Chidume and Ofoedu \cite{CO1}, Chidume and Zegeye \cite{b25}, Diop et al. \cite{b40}).
\par In this paper, the goal is to establish a strong convergence theorem for nonlinear Hammerstein equation Eq. (\ref{e3}), using a technique which does not require the assumption of existence of a real constant whose calculation is unclear. This study considers nonlinear integral equations of Hammerstein type with $(p, {\eta})$-strongly monotone mapping, $p>1,$ and ${\eta} \in (1, \infty).$ Let $E$ be a Banach space with dual $E^*$ and define $X:=E\times E^*.$ Let $p>1,$ ${\eta}_1, {\eta}_1 \in (1, \infty)$ and suppose $F: E\rightarrow E^*$ is a $(p, {\eta}_1)$-strongly monotone mapping and $K: E^*\rightarrow E$ is a $(p, {\eta}_2)$-strongly monotone mapping such that $D(K)=R(F)=E^*.$ An arbitrary mapping $A: X\rightarrow X^*,$ is defined in term of $F$ and $K.$ Suitable conditions are imposed to show that $A$ is a $(p, {\eta})$-strongly monotone mapping with ${\eta}=\min\left\{{\eta}_1, {\eta}_2\right\}$ and to obtain a strong convergence result. The forced oscillations of finite amplitude of a pendulum is shown as a specific example of nonlinear integral equations of Hammerstein type. Numerical examples are also given for the illustration of the convergence of the sequences of iteration.  These show the application of our results in solving the problems which occur in physical sciences.

\section{Preliminaries}
\begin{definition}
Let $E$ be a real Banach space and $S := \left\{x \in E : \|x \| = 1\right\}$. $E$ is said to have a \textit{G}$\hat{a}$\textit{teaux differentiable norm} if the limit
\begin{equation}\label{chp21}
\displaystyle \lim_{t\rightarrow 0}\frac{\|x+ty \|-\|x \|}{t}
\end{equation}
exists for each $x, y \in S.$ A Banach space $E$ is said to be \textit{smooth} if for every $x\neq 0$ in $E,$ there is a unique $x^*\in E^*$ such that $\|x^*\|=1$ and $\left\langle x, x^*\right\rangle = \|x\|,$ where $E^*$ denotes the dual of $E.$ $E$ is said to be \textit{uniformly smooth} if it is smooth and the limit (\ref{chp21}) is attained uniformly for each $x, y \in S.$ 
\end{definition}
\begin{definition}
The \textit{modulus of convexity} of a Banach space $E$, $\delta_{E}: (0, 2]\rightarrow [0, 1]$ is defined by
 $$\delta_{E}(\epsilon)=\inf \left\{1-\frac{\|x + y \|}{2}  : \|x \|=\|y \|=1, \|x - y \| > \epsilon \right\}.$$
$E$ is \textit{uniformly convex} if and only if ${\delta}_E(\epsilon) > 0$ for every $\epsilon \in(0, 2]$. A normed linear space $E$ is said to be \textit{strictly convex} if
$$\|x \|=\|y \|=1, x\neq y\Rightarrow \frac{\|x + y \|}{2}<1.$$
 It is well known that a space $E$ is uniformly smooth if and only if $E^*$ is uniformly convex. 
 \end{definition}
\begin{definition}
Let $X$ and $Y$ be Banach spaces and $A: X\rightarrow Y$ be a mapping. $A$ is \textit{uniformly continuous} if for each $\epsilon >0,$ there exists $\delta >0$ such that 
	$$\forall~ x, y \in X \ \mbox{with} \ \|x-y\|<\delta \ \mbox{we have} \ \|Ax-Ay\|<\epsilon.$$
 Let ${\psi}(t)$ be a function on the set ${\R}^+$ of nonnegative real numbers such that:
\begin{itemize}
	\item [(i)] ${\psi}$ is nondecreasing and continuous;
	\item [(ii)] ${\psi}(t)=0$ if and only if $t=0.$
\end{itemize}
$A$ is said to be uniformly continuous if it admits the modulus of continuity ${\psi}$ such that 
$$\|A(x)- A(y)\|\leq{\psi}(\|x- y\|) \ \forall~ x, y\in X.$$ 
The modulus of continuity ${\psi}$ has some useful properties which are listed below (for instance, see e.g,  Altomare et al. \cite{ftc1}, pp. 266-269, Forster \cite{ftc2}, Ipatov \cite{ftc3}):
\begin{itemize}
	\item [(a)] Modulus of continuity is subadditive: For all real numbers $t_1\geq 0, t_2\geq 0,$ we have 
	$${\psi}(t_1+t_2)\leq{\psi}(t_1)+{\psi}(t_2).$$
	\item [(b)]  Modulus of continuity is monotonically increasing: If $0\leq t_1\leq t_2$ holds for some real numbers $t_1, t_2,$ then
	$$0\leq {\psi}(t_1)\leq {\psi}(t_2).$$
	\item [(c)]  Modulus of continuity is continuous: The modulus of continuity ${\psi}:{\R}^+\rightarrow {\R}^+$ is continuous on the set positive real numbers, in particular, the limit of ${\psi}$ at $0$ from above is
	$$\displaystyle \lim_{t\rightarrow 0}{\psi}(t)=0.$$
\end{itemize}
\end{definition}
\begin{definition}
A set $B$ is said to be compact if every open cover of $B$ has a finite subcover. That is, if $\left\{U_n\right\}_{n=1}^{\infty}$  is a collection of open sets such that $B\subseteq \bigcup_{n=1}^{\infty}U_n,$ then there is a finite subcollection $\left\{U_{n_k}\right\}_{k=1}^K$ such that $B\subseteq \bigcup_{k=1}^KU_{n_k}.$
 \end{definition}
\begin{definition}
 Let $\nu: [0, \infty )\rightarrow [0, \infty )$ be a continuous, strictly increasing function such that $\nu (t)\rightarrow \infty$ as $t\rightarrow \infty$ and $\nu(0)=0$ for any $t\in [0, \infty).$ Such a function $\nu$ is called a \textit{gauge function}. A \textit{duality mapping} associated with the guage function $\nu$ is a mapping $J^E_{\nu} : E\rightarrow 2^{E^*}$ defined by
$$J^E_{\nu}(x) =\left\{ f \in E^* :\left\langle x, f \right\rangle = {\|x\|} \nu(\| x\|), \| f\|=\nu(\| x\|)\right\},$$
where $\left\langle . , .\right\rangle$ denotes the duality pairing. For $p>1,$ let $\nu(t)=t^{p-1}$ be a gauge function. $J^E: E\rightarrow 2^{E^*}$ is called a \textit{generalized duality mapping} from $E$ into $2^{E^*}$ and is given by 
$$J^E(x) =\left\{ f \in E^* :\left\langle x, f \right\rangle = {\|x\|}^p, \| f\|={\| x\|}^{p-1} \right\}.$$
$J^E$ is uniformly continuous on bounded subsets of $E.$ For $p = 2,$ the mapping $J_2^E$ is called the normalized duality mapping. In a Hilbert space, the normalized duality mapping is the identity map.
\par The following results about the generalized duality mappings are well known and have been established in a number of works (see e.g., Alber and Ryazantseva \cite{b2}, p. 36, Cioranescu \cite{b3}, p. 25-77, Xu and Roach \cite{b7}, Z$\check{a}$linescu  \cite{b37}). Let $E$ be a Banach space, then,
\begin{itemize}
\item [(i)] $E$ is smooth if and only if $J^E$ is single-valued;
\item [(ii)] If $E$ is reflexive, then $J^E$ is onto;
\item [(iii)] If $E$ has uniform G$\hat{a}$teaux differentiable norm, then $J^E$ is norm-to-weak$^*$ uniformly continuous on bounded sets;
\item [(iv)] $E$ is uniformly smooth if and only if $J^E$ is single valued and uniformly continuous on any bounded subset of $E;$
\item [(v)] If $E$ is strictly convex, then $J^E$ is one-to-one, that is, $\forall~ x, y \in E, ~x\neq y\Rightarrow J^E(x)\cap J^E(y)=\emptyset;$
\item [(vi)] If $E$ and $E^*$ are strictly convex and reflexive, then $J^{E^*}$ is the generalized duality mapping from $E^*$ to $E$ and $J^{E^*}$ is the inverse of $J^E;$
\item [(vii)] If $E$ is uniformly smooth and uniformly convex, the generalized duality mapping $J^{E^*}$ is uniformly continuous on any bounded subset of $E^*;$
\item [(viii)] If $E$ and $E^*$ are strictly convex and reflexive, for all $x\in E$ and $f\in E^*,$ the equalities $J^EJ^{E^*}f=f$ and $J^{E^*}J^Ex=x$ hold.
\end{itemize}
\end{definition}
\begin{definition}
Let $C$ be a nonempty subset of $E$ and $A$ be a mapping from $C$ into itself. Then,
\begin{itemize}
	\item [(i)] $A$ is nonexpansive provided $\|Ax-Ay\|\leq \|x-y\| \ \mbox{for all} \ x, y\in C;$
	\item [(ii)] $A$ is firmly nonexpansive type (see e.g., Kohsaka and Takahashi \cite{b32}) if $\left\langle Ax-Ay, J^EAx-J^EAy\right\rangle \leq \left\langle Ax-Ay, J^Ex-J^Ey\right\rangle $ for all $ x, y\in C$.
\end{itemize}
\end{definition}	
\begin{definition}
Let $E$ be a smooth Banach space and $A :E \rightarrow E^*,$ be a mapping with $x,y\in E.$ $A$ is said to be monotone $${\left\langle x - y, Ax - Ay \right\rangle} \geq 0.$$
  $A$ is called strongly monotone if
 $$\left\langle x-y, Ax - Ay \right\rangle \geq k{\|x-y \|}^2,$$
 where $k$ is a positive constant (Alber and Ryazantseva \cite{b2}, page 25).
Let $p>1,$ $A$ is said to be $(p, k)$-\textit{strongly monotone} if $${\left\langle  x - y, Ax - Ay \right\rangle} \geq k {\|x-y \|}^p,$$
for a constant $k > 0$ (Chidume and Djitte \cite{r5} and Chidume and Shehu \cite{r6}).
   \begin{remark}
   \end{remark}
   According to definition of Chidume and Djitte \cite{r5} and Chidume and Shehu \cite{r6}, a strongly monotone mapping is referred to as a $(2, k)$-strongly monotone mapping.
\par $A$ is called \textit{maximal monotone} if it is monotone and its graph is not properly contained in the graph of any other monotone mapping. As a result of Rockafellar \cite{b5}, it follows that $A$ is maximum monotone if it is monotone and the range of $(J^E+tA)$ is all of $E^*$ for some $t>0.$ 
Let $E$ be a reflexive smooth strictly convex space and $A$ be a mapping such that the range of $(J^E+tA)$ is all of $E^*$ for some $t>0$ and let $x\in E$ be fixed.  Then for every $t>0,$ there corresponds a unique element $x_t \in D(A)$ such that
\begin{equation}
  J^Ex= J^Ex_t+tAx_t.
\end{equation}
Therefore, the \textit{resolvent} of $A$ is defined by $J^A_tx=x_t.$ In other words, $J^A_t=(J^E+tA)^{-1}J^E$ and $A^{-1}0=F(J^A_t)$ for all $t>0,$ where $F(J^A_t)$ denotes the set of all fixed points of $J^A_t.$ The resolvent $J^A_t$ is a single-valued mapping from $E$ into $D(A)$ (Kohsaka and Takahashi \cite{kt}).
\end{definition}
\begin{definition}\label{d1}
 Alber \cite{b1} introduced the functions ${\phi}: E \times E \rightarrow \R,$ defined by 
\begin{equation}
{\phi} (x,y)= {\|x \|}^2 - 2\left\langle x, J_2^Ey \right\rangle + {\|y \|}^2, \ \mbox{for all} \ x,y \in E,
\end{equation}
where $J_2^E$ is the normalized duality mapping from $E$ to $E^*.$ Let $E$ be a smooth real Banach space and $p, q>1$ with $\frac{1}{p} + \frac{1}{q} =1.$ Aibinu and Mewomo \cite{b4} introduced the functions ${\phi}_p : E \times E \rightarrow \R,$ defined by
$${\phi}_p (x,y)=  \frac{p}{q}{\|x \|}^q - p\left\langle x, J^Ey \right\rangle + {\|y \|}^p , \ \mbox{for all} \ x,y \in E$$
and $V_p: E\times E^* \rightarrow \R, $ defined as 
$$V_p (x,x^*)= \frac{p}{q}{\|x \|}^q - p\left\langle x, x^* \right\rangle + {\|x^* \|}^p  ~~ \forall ~~ x \in E, x^*\in E^*,$$
where $J^E$ is the generalized duality mapping from $E$ to $E^*.$ 
\end{definition}
\begin{remark}
We have the following remarks follow from Definition \ref{d1} (Alber \cite{b1}, Aibinu and Mewomo \cite{b4}):
\begin{itemize}
\item[(i)] For all $x,y \in E,$
\begin{equation}\label{e11}
(\|x\|- \|y \|)^p \leq {\phi}_p (x,y) \leq (\|x\| + \|y \|)^p.
\end{equation}
\item [(ii)] It is obvious that
\begin{equation}\label{e20}
V_p (x,x^*)= {\phi}_p (x,J^{E^*}x^*)~~ \forall ~~ x \in E,\ \ x^*\in E^*.
\end{equation}
\end{itemize}
\end{remark}
\begin{definition}
 Let $E$ be a topological real vector space and $T$ a multivalued mapping from $E$ into $2^{E^*}.$ Cauchy-Schwartz's inequality is given by
\begin{equation}
 |\left\langle x, y^*\right\rangle| \leq {\left\langle x, x^*\right\rangle}^{\frac{1}{2}}{\left\langle y, y^*\right\rangle}^{\frac{1}{2}},
\end{equation}
for any $x$ and $y$ in $D(T)$ and any choice of $x^* \in Tx$ and $y^* \in Ty$ (Zarantonello \cite{chp24}). 
\end{definition}
\par In the sequel, we shall need the lemmas whose proofs have been established (See e.g, Alber \cite{b1}, Aibinu and Mewomo \cite{b4}).
\begin{lemma}\label{l20}
 Let $E$ be a strictly convex and uniformly smooth real Banach space and $p>1.$ Then
\begin{equation}\label{e21}
V_p (x, x^*) + p\left\langle J^{E^*}x^*-x, y^* \right\rangle \leq V_p(x, x^*+y^*)
\end{equation}
for all $x\in E$ and $x^*, y^* \in E^*.$
\end{lemma}
\begin{lemma}\label{l21}
Let $E$ be a smooth uniformly convex real Banach space and $p > 1$ be an arbitrarily real number. For $d > 0$, let $ B_d(0):= \left\{ x \in E: \| x \| \leq d \right\} $. Then for arbitrary $x, y \in B_d(0)$,
$$ {\|x-y \|}^p \geq{\phi}_p (x,y)- \frac{p}{q}{\| x\|}^q, \ \mbox{where} \ \frac{1}{p}+\frac{1}{q}=1.$$
\end{lemma}

\begin{lemma}\label{l31}
Let $E$ be a reflexive strictly convex and smooth real Banach space and $p>1.$ Then
\begin{equation}\label{e31}
{\phi}_p(y,x)-{\phi}_p(y,z)\geq  p\left\langle z-y, J^Ex-J^Ez\right\rangle= p\left\langle y-z, J^Ez-J^Ex\right\rangle  \ \mbox{ for all} \ x, y, z\in E.
\end{equation}
\end{lemma}

\begin{lemma}\label{t3}
Let $E$ be a real uniformly convex Banach space. For arbitrary $r>0$, let $B_r(0):=\left\{x \in E: \|x\| \leq r\right\}$. Then, there exists a continuous strictly increasing convex function
$$g:[0,\infty)\rightarrow [0,\infty),~~ g(0)=0,$$
such that for every $x, y \in B_r(0), j^E(x)\in J^E(x), j^E(y)\in J^E(y)$, we have
$\left\langle x-y, j^E(x)-j^E(y) \right\rangle \geq g(\|x-y \|)$ (See Xu \cite{b18}).
 \end{lemma}
\begin{lemma}\label{l11}
Let $\left\{a_n\right\}$ be a sequence of nonnegative real numbers satisfying the following
relations:
$$a_{n+1}\leq(1-{\alpha}_n)a_n+{\alpha}_n{\sigma}_n+{\gamma}_n, ~~n\in \N,$$
where
\begin{itemize}
\item[(i)]${\left\{\alpha\right\}}_n\subset (0, 1)$, $\displaystyle\sum_{n=1}^{\infty} {\alpha}_n = \infty$;
\item[(ii)]$\limsup {\left\{\sigma\right\}}_n\leq 0$;
\item[(iii)] ${\gamma}_n\geq 0$, $\displaystyle\sum_{n=1}^{\infty} {\gamma}_n < \infty$.
\end{itemize}
 Then, $a_n\rightarrow 0$ as $n\rightarrow \infty$ (See Xu \cite{bh1}).
\end{lemma}

\begin{lemma}\label{l15}
Let $E$ be a smooth uniformly convex real Banach space and let $\left\{x_n\right\}$ and $\left\{y_n\right\}$ be two sequences from $E.$ If either $\left\{x_n\right\}$ or $\left\{y_n\right\}$ is bounded and $\phi (x_n, y_n) \rightarrow 0$ as $n \rightarrow \infty$, then $ \| x_n - y_n \| \rightarrow 0$ as $n \rightarrow \infty$ (See Kamimura and Takahashi \cite{r16}).
\end{lemma}

\begin{lemma}\label{l16}
For a real number $p>1$, let $X, Y$ be real uniformly convex and uniformly smooth spaces. Let $Z:=X\times Y$ with the norm ${\|z\|}_Z=\left({\|u\|}^p_X+{\|v\|}^p_Y \right)^{\frac{1}{p}}$ for arbitrary $z:=(u,v)\in Z$. Let $Z^*:=X^*\times Y^*$ denotes the dual space of $Z$. For arbitrary $z=(u,v)\in Z$, define the map $J^Z:Z\rightarrow Z^*$ by
$$J^Z(z)=J^Z(u,v)=\left(J^X(u),J^Y(v)\right),$$
such that for arbitrary $z_1=(u_1,v_1)$, $z_2=(u_2,v_2)$ in $Z$, the duality pairing $\left\langle.,.\right\rangle$ is given by
$$\left\langle z_1,J^Z(z_2)\right\rangle=\left\langle u_1,J^X(u_2)\right\rangle+\left\langle v_1,J^Y(v_2)\right\rangle.$$
Then (see Chidume and Idu \cite{b28}),
\begin{itemize}
	\item[(i)]$Z$ is uniformly smooth and uniformly convex,
	\item[(ii)]	$J^Z$ is single-valued duality mapping on $Z.$
\end{itemize}
\end{lemma}

\section{Results}
\begin{definition}
Let $E$ be a real smooth Banach space with dual space $E^*.$ Let $X := E\times E^*$  and define ${\wedge}_p : X\times X \rightarrow \R$ by
$$ {\wedge}_p (x_1 , x_2) = {\phi}_p (u_1, u_2) + {\phi}_p(v_1, v_2) ~~\forall ~ x_1, x_2 \in X ,$$
where respectively $x_1=(u_1, v_1)$ and $x_2=(u_2, v_2)$.
\end{definition}
We first give and prove the following lemmas which are useful in establishing our main results.
\begin{lemma}\label{l7} 
Let $E$ be a uniformly smooth and uniformly convex real Banach space with the dual $E^*.$ Let $p>1,$ $\eta_1, \eta_1 \in (1, \infty)$ and suppose $F: E\rightarrow E^*,~~ K: E^*\rightarrow E$ are respectively $(p, \eta_1)$-strongly monotone and $(p, \eta_2)$-strongly monotone mappings such that $D(K) = R(F ) = E^*.$ For a real number $p>1,$ let $X := E\times E^*$ with norm ${\|x \|}_X :=  \left( {\| u\|}_{E}^p + {\| v \|}_{E^*}^p  \right)^{\frac{1}{p}} ~ \forall ~ x = \left(u, v\right) \in X$ and the dual is denoted by $X^* := E^*\times E.$  Define a mapping $A: X\rightarrow X^*$ by
\begin{equation}
Ax = \left(Fu - v, Kv + u \right), \forall ~ x=(u, v)\in X.
\end{equation}
Then $A$ is a $(p, \eta)$-strongly monotone mapping, where $\eta := \min \left\{{\eta}_1,{\eta}_2\right\}$.
\end{lemma}
\begin{proof}
Let $x_1 = (u_1, v_1), ~~x_2 = (u_2, v_2) \in X$. We have $Ax_1 = A(u_1, v_1) = \left(Fu_1 - v_1, Kv_1 + u_1 \right)$ and
 $Ax_2 = A(u_2, v_2) = \left(Fu_2 - v_2, Kv_2 + u_2 \right)$ such that $$ Ax_1 - Ax_2 = \left( Fu_1 - Fu_2 - (v_1 -v_2), Kv_1 - K v_2 + (u_1 - u_2)  \right).$$ Since $F$ and $K$ are strongly monotone with $\eta_1$ and $\eta_2,$ respectively as constants of strong monotonicity, we obtain,
 \begin{eqnarray*}
 \left\langle x_1 - x_2, Ax_1 - Ax_2  \right\rangle &=& \left\langle u_1 -u_2, Fu_1 - Fu_2 - (v_1 -v_2) \right\rangle \\
& &+ \left\langle Kv_1 - K v_2 + (u_1 - u_2),  v_1 -v_2 \right\rangle \\
&=& \left\langle u_1 -u_2, Fu_1 - Fu_2\right\rangle + \left\langle u_1 -u_2, - (v_1 -v_2) \right\rangle \\
& &+  \left\langle v_1 -v_2,  Kv_1 - K v_2 \right\rangle +  \left\langle v_1 -v_2,  u_1 - u_2 \right\rangle \\
&\geq &{\eta}_1{\| u_1-u_2 \|}^p + { \eta}_2{\| v_1-v_2 \|}^p\\
&\geq & \min \left\{{\eta}_1,{\eta}_2\right\}({\| u_1-u_2 \|}^p+{\| v_1-v_2 \|}^p) \\
&= & \eta{\| x_1-x_2 \|}^p.
\end{eqnarray*}
Hence,  $A$ is a $(p, \eta)$-strongly monotone mapping.
\end{proof}
\begin{lemma}\label{pdf1} 
Let $A: B\rightarrow \R$ be continuous and $B$ be compact subset of $E.$ Then there exists some $T\in \R$ such that $|A(x)| < T$ for all $x \in B.$
\end{lemma}
\begin{proof}
For contradiction suppose that $A$ is unbounded. By definition of continuity, the set
$$U_n=A^{-1}\left(\left(-n, n\right)\right)=\left\{x:|A(x)| < n\right\}$$
is open. Then
$$\R = \displaystyle \bigcup_{n=1}^{\infty}\left(-n, n\right)\Rightarrow B\subseteq \displaystyle \bigcup_{n=1}^{\infty}A\left(\left(-n, n\right)\right)=\displaystyle \bigcup_{n=1}^{\infty}U_n.$$
For $x\in B,~ A(x)\in U_n$ for some $n,$ which implies that $|A(x)|< n.$ However $B$ is compact, therefore there is a finite subcollection $\left\{U_{n_k}\right\}_{k=1}^K$ which is still a cover for $k.$ Define
$$T:= \max \left\{n_1, n_2, ..., n_K\right\},$$
and notice that for any $x\in B,$ it must be that
$$A(x)\in U_{n_k} \ \mbox{for some}\ \ k\leq K \Rightarrow |A(x)|< n_k \leq T.$$
Hence, $A$ is bounded.
\end{proof}
\begin{lemma}\label{l6} 
Let $E$ be a uniformly smooth and uniformly convex real Banach space with the dual $E^*.$ Let $p>1,$ $\eta_1, \eta_1 \in (1, \infty)$ and suppose $F: E\rightarrow E^*$ is a continuous $(p, \eta_1)$-strongly monotone mapping such that the range of $(J^E+t_1F)$ is all of $E^*$ and $K: E^*\rightarrow E$ is a continuous $(p, \eta_2)$-strongly monotone mapping such that the range of $(J^{E^*}+t_2K)$ is all of $E$ for some $t_1, t_2>0$ with $D(K) = R(F ) = E^*.$ Let $X := E\times E^*$ with norm ${\|x \|}_X :=  \left( {\| u\|}_{E}^p + {\| v \|}_{E^*}^p  \right)^{\frac{1}{p}} ~ \forall ~ x = \left(u, v\right) \in X$ and define a mapping $A: X\rightarrow X^*$ by
\begin{equation}
Ax = \left(Fu - v, Kv + u \right), \forall ~ x=(u, v)\in X.
\end{equation}
Then $A$ is a continuous $(p, \eta)$-strongly monotone mapping such that the range of $(J^X + tA)$ is all of $X^*$ for some $t > 0,$ where $\eta := \min \left\{{\eta}_1,{\eta}_2\right\}.$
\end{lemma}

\begin{proof}
 We show that $R(J^X + t A)=X^*$ for some $t > 0.$ Without loss of generality, let ${t}_0$ be such that $0 < {t}_0 < 1.$
Since $F$ is such that $R(J^E+{t}_0F)=E^*$ and by the strict convexity of $E,$ we obtain for every $u \in E,$ there exists a unique $u_{{t}_0} \in E$ such that
$$J^Eu = J^Eu_{{t}_0} + {t}_0 Fu_{{t}_0}.$$
Define $J^E_{{t}_0} u := u_{{t}_0},$ in other words, define a single-valued mapping $J^{E}_{{t}_0} : E\rightarrow D(F)$ by $J^{E}_{{t}_0} = (J^E + {t}_0 F)^{-1}J^E,$ where $D(F)$ is the domain of $F.$ Such a $J^{E}_{{t}_0} $ is called the resolvent of $F.$ Similarly, since $K$ is such that $R(J^{E^*}+{t}_0K)=E$ and by strict convexity of $E^*,$ define the resolvent $J^{E^*}_{{t}_0}: E^* \rightarrow D(K)$ by  $J^{E^*}_{{t}_0} = (J^{E^*} + {t}_0 K)^{-1}J^{E^*},$ where $D(K)$ is the domain of $K.$ It is known that $(J^E + {t}_0 F)$ and $(J^{E^*} + {t}_0 K)$ are bijections (see e.g, Chuang \cite{r14}). It can be easily verified that if $E$ is a smooth, strictly convex and reflexive Banach space and $F: E \rightarrow E^*$ is a monotone mapping with $R(J^E+t F)=E^*,$ then for each $t >0,$ the resolvent $J^E_{t}$ of $F$ defined by 
$$J_{t}^Eu=\left\{z\in E:J^Eu = J^Ez+ t Fz \right\}=\left\{\left(J^E+t F\right)^{-1}J^Eu\right\}$$
for all $u\in E$ is a firmly nonexpansive type mapping. Indeed, for each $u, v\in E$, $t >0$ and for every $J^E_{t}u,  J^E_{t}v \in E,~~~ \frac{J^Eu-J^E(J^E_{t}u)}{t}, \frac{J^Ev-J^E(J^E_{t}v)}{t} \in F$ and by the monotonicity of $F$, we have
$$\left\langle J^E_{t}x-J^E_{t}v, \frac{J^Eu-J^E(J^E_{t}u)}{t}-\frac{J^Ev-J^E(J^E_{t}v)}{t}\right\rangle \geq0.$$
Such that
$$\left\langle J^E_{t}u-J^E_{t}v, J^E(J^E_{t}x)-J^E(J^E_{t}v) \right\rangle \leq  \left\langle J^E_{t}x-J^E_{t}v,J^Eu-J^Ev \right\rangle.$$
Therefore, for $h :=(h_1, h_2) \in X$, define $G: X \rightarrow X$ by $$Gx = \left(J^{E^*}_{{t}_0}(h_2+ {t}_0 v), J^{E}_{{t}_0}(h_1 - {t}_0 u)\right), \forall ~~ x = (u, v)\in X.$$
By the Lipschitz continuity property of $J^{E},$ we have
$$\| Gx_1 - Gx_2 \| \leq {t}_0 \| x_1 -x_2 \|~~~~ \forall~~~~ x_1, x_2 \in X.$$
Therefore $G$ is a contraction. So by the Banach contraction mapping principle, $G$ has a unique fixed point $x^*:=(u^*, v^*)\in X$, that is $Gx^* = x^*$ or equivalently $u^* = J^{*K}_{{t}_0} (h_2+ {t}_0 v^*)$ and  $v^* = (J^{F}_{{t}_0} (h_1 - {t}_0 u^*).$ These imply $(J^X + {t}_0 A)x=h.$ Therefore, $R(J^X + {t}_0 A)=X^*.$ For $t\geq 1,$ it is obtained by similar analysis that $R(J^X + {t}A)=X^*$ (see e.g., Chidume and Shehu \cite{r6}). Recall that $A$ is $(p, \eta)$-strongly monotone by Lemma \ref{l7}. Then $A$ is an $(p, \eta)$-strongly monotone mapping with $R(J^X + t A)=X^*$ for some $t > 0.$ 
\end{proof}

\par Next is to give our main theorem.
\begin{theorem}\label{t5}
Let $E$ be a uniformly smooth and uniformly convex real Banach space with the dual $E^*.$  Let $p>1,$ $\eta_1, \eta_1 \in (1, \infty)$ and suppose $F: E\rightarrow E^*$ is a continuous $(p, \eta_1)$-strongly monotone mapping such that the range of $(J^E+t_1F)$ is all of $E^*$ and $K: E^*\rightarrow E$ is a continuous $(p, \eta_2)$-strongly monotone mapping such that the range of $(J^{E^*}+t_2K)$ is all of $E$ for some $t_1, t_2>0$ with $D(K) = R(F ) = E^*.$ For arbitrary $u_1 \in E$ and $v_1\in E^*$, let $\left\{u_n\right\}$ and $\left\{v_n\right\}$ be the sequences defined iteratively by
\begin{equation}\label{am1}
u_{n+1} = J^{E^*}\left(J^Eu_n - {\lambda}_n\left(Fu_n-v_n+{\theta}_n(J^Eu_n-J^Eu_1)\right)\right), n \in \N,
\end{equation}
\begin{equation}\label{am2}
v_{n+1} = J^E\left(J^{E^*}v_n - {\lambda}_n\left(Kv_n+u_n+{\theta}_n(J^{E^*}v_n-J^{E^*}v_1)\right)\right), n \in \N,
\end{equation}
where $J^E:E\rightarrow E^*$ is the generalized duality mapping with the inverse, $J^{E^*}:E^*\rightarrow E$ and the real sequences $\left\{ {\lambda}_n\right\}^{\infty}_{n=1}\subset (0,1)$ and $\left\{ {\theta}_n\right\}^{\infty}_{n=1}$ in $(0,\frac{1}{2})$ are such that,
\begin{itemize}
	\item[(i)] $\displaystyle \lim_{n\rightarrow \infty} {\theta}_n =0;$ 
	\item[(ii)]$ \displaystyle\sum_{n=1}^{\infty} {\lambda}_n{\theta}_n=\infty$; ${\lambda}_n=o({\theta}_n)$;
	\item[(iii)]$\displaystyle \lim_{n\rightarrow \infty}\left(({\theta}_{n-1}/{\theta}_n)-1\right)/{{\lambda}_n{\theta}_n}=0.$
\end{itemize}	
 Suppose that $0 = u + KFu$ has a solution in $E.$ Then the sequence $\left\{u_n \right\}$ converges strongly to the solution of $u+KFu=0.$
\end{theorem}

\begin{proof}
Let $X := E\times E^*$ with norm ${\|x \|}_X^p :=   {\| u\|}_{E}^p + {\| v \|}_{E^*}^p   ~ \forall ~ x = \left(u, v\right) \in X.$ Define the sequence $\left\{x_n\right\}$ in $X$ by $x_n := (u_n, v_n).$ Let $u^*\in E$ be a solution of $u+KFu = 0.$ Observe that setting $v^*:=Fu^*$ and $x^*:=(u^*, v^*),$ we have that $u^* = -Kv^*.$

 We divide the proof into two parts.\\
{\bf Part 1:} We prove that $\left\{x_n \right\}$ is bounded. Let $p>1$ with $\frac{1}{p}+\frac{1}{q}=1$ and $u^*\in E$ be a solution of the equation $0 = u + KFu.$ It suffices to show that ${\phi}_p(x^*, x_n)\leq r, ~\forall ~n\in \N.$ The proof is by induction. Let $r > 0$ be sufficiently large such that:
\begin{equation}\label{e22}
r\geq \max \left\{ {\phi}_p( x^*, x_1), 4MQ, \frac{4p}{q}{\| x\|}^q \right\},
\end{equation}
where $M>0$ and $Q>0$ are arbitrary but fixed. For $n =1,$ we have by construction that ${\phi}_p( x^*, x_1)\leq r$ for real $p>1.$ Assume that ${\phi}_p( x^*, x_n)\leq r$ for some $n\geq 1.$ From inequality (\ref{e11}), we have $\|x_n \| \leq r^{\frac{1}{p}} + \|x^* \|$ for real $p>1.$ The next task is to show that  ${\phi}_p( x^*, x_{n+1})\leq r.$ Let $B:=\left\{z\in E: {\phi}_p(x^*, z)\leq r \right\}.$ $J^E$ is known to be uniformly continuous on bounded subsets of $E$ and $F$ is bounded on $B$ (Lemma \ref{pdf1}). Define
$$M_1 :=\sup\left\{ {\|Fu_n-v_n+{\theta}_n(J^Eu_n-J^Eu_1)\|}: {\theta}_n\in(0,1), u_n\in B \right\}+1.$$
Let ${\psi}_1$ denotes the modulus of continuity of $J^{E^*}$ and observe that $J^{E^*}(J^Eu_n)=u_n.$ Then
\begin{eqnarray}\label{e30}
\|u_n-u_{n+1}\|&=&\|x_n-  J^{E^*}( J^Eu_n - {\lambda}_n\left(Fu_n-v_n+{\theta}_n(J^Eu_n-J^Eu_1)\right)) \|\nonumber\\
&=&\|J^{E^*}(J^Ex_n)-  J^{E^*}(J^Ex_n - {\lambda}_n\left(Fu_n-v_n+{\theta}_n(J^Eu_n-J^Eu_1)\right)) \| \nonumber\\
&\leq&{\psi}_1\left( |{\lambda}_n|\|Fu_n-v_n+{\theta}_n(J^Eu_n-J^Eu_1)\| \right)\nonumber\\
&\leq&{\psi}_1\left( |{\lambda}_n|M_1\right)\nonumber\\
&\leq& {\psi}_1 \left( \sup\left\{ |{\lambda}_n|M_1: {\lambda}_n\in(0,1)\right\} \right).
\end{eqnarray}
The $\sup\left\{ |{\lambda}_n|M_1\right\}$ exists and it is a real number different from infinity due to the boundedness of $F$ and uniform continuity of $J^E$ on bounded subsets of $E.$ Let $Q_1=:{\psi}_1 \left(\sup \left\{ |{\lambda}_n|M_1\right\}\right).$ Similarly, let ${\psi}_2$ be the modulus of continuity of $J^E: E\rightarrow E^*$ on bounded subsets of $E$ and observe that $J^E(J^{E^*}v_n)=v_n.$ Define
$$M_2 :=\sup\left\{ {\|Kv_n+u_n+{\theta}_n(J^{E^*}v_n-J^{E^*}v_1)\|}: {\theta}_n\in(0,1),  v_n\in B  \right\}+1,$$
then
\begin{eqnarray}\label{e32}
\|v_n-v_{n+1}\|&\leq&{\psi}_2\left( \sup\left\{ |{\lambda}_n|M_2: {\lambda}_n\in(0,1)\right\} \right).
\end{eqnarray}
Let $Q_2=:{\psi}_2 \left(\sup \left\{ |{\lambda}_n|M_2\right\}\right)$ and define $M:=M_1+M_2$ and $Q:=Q_1+Q_2.$
Let $\eta:=\min\left\{{\eta}_1, {\eta}_2\right\},$ by applying Lemma \ref{l20} with $y^* := {\lambda}_n\left(Fu_n-v_n+{\theta}_n(J^Eu_n-J^Eu_1)\right)$ and by using the definition of $u_{n+1}$, we compute as follows,
  \begin{eqnarray*}
{\phi}_p(u^*, u_{n+1})
 & = & {\phi}_p\left(u^*, J^{E^*}\left(J^Eu_n - y^*\right)\right)\\
 & = & V_p\left(u^*, J^Eu_n - {\lambda}_n\left(Fu_n-v_n+{\theta}_n(J^Eu_n-J^Eu_1)\right)\right)\\
 & \leq & V_p(u^*, J^Eu_n )\\
 &&-p{\lambda}_n\left\langle J^{E^*}( J^Eu_n - y^*)-u^*, Fu_n-v_n+{\theta}_n(J^Eu_n-J^Eu_1)\right\rangle \\
 &= & {\phi}_p(u^*, u_n )-p{\lambda}_n \left\langle u_n-u^*,Fu_n-v_n+{\theta}_n(J^Eu_n-J^Eu_1)\right\rangle\\
 & &-p{\lambda}_n\left\langle J^{E^*}( J^Eu_n - y^*)-u_n, Fu_n-v_n+{\theta}_n(J^Eu_n-J^Eu_1)\right\rangle.
 \end{eqnarray*}
 By Schwartz inequality and by applying inequality (\ref{e30}), we obtain
\begin{eqnarray*}
{\phi}_p(u^*, u_{n+1})
& \leq & {\phi}_p(u^*, u_n )-p{\lambda}_n \left\langle u_n-u^*,Fu_n-v_n+{\theta}_n(J^Eu_n-J^Eu_1)\right\rangle + p{\lambda}_nM_1Q_1\\
& = &{\phi}_p(u^*, u_n )-p{\lambda}_n \left\langle u_n-u^*,Fu_n-Fu^*\right\rangle  \ \mbox{(since $u^*\in N(F)$)} \\
& &-p{\lambda}_n \left\langle u_n-u^*,v^*-v_n\right\rangle-p{\lambda}_n{\theta}_n \left\langle u_n-u^*, J^Eu_n-J^Eu_1\right\rangle +p{\lambda}_nM_1Q_1.
\end{eqnarray*}
By Lemma \ref{l31}, $p\left\langle u_n-u^*, J^Eu_1-J^Eu_n\right\rangle \leq {\phi}_p(u^*, u_1 ) - {\phi}_p(u^*, u_n ).$ Consequently,\\ $p\left\langle u_n-u^*, J^Eu_1-J^Eu_n\right\rangle \leq {\phi}_p(u^*, u_1 ).$  Therefore, using $(p, {\eta}_1)$-strongly monotonicity property of $F,$ we have,
\begin{eqnarray}\label{e10}
{\phi}_p(u^*, u_{n+1})
 &\leq&  {\phi}_p(u^*, u_n )-p{\eta}_1{\lambda}_n{\|u_n -u^* \|}^p-p{\lambda}_n \left\langle u_n-u^*,v^*-v_n\right\rangle \nonumber\\ 
  &&+p{\lambda}_n{\theta}_n p\left\langle u_n-u^*, Ju_1-Ju_n\right\rangle +p{\lambda}_nM_1Q_1 \nonumber\\
 &\leq&  {\phi}_p(u^*, u_n )-p{\lambda}_n{\|u_n -u^* \|}^p-p{\lambda}_n \left\langle u_n-u^*,v^*-v_n\right\rangle \nonumber\\
 &&+p{\lambda}_n{\theta}_n {\phi}_p(u^*, u_1 )+p{\lambda}_nM_1Q_1 \nonumber \\
 &\leq&  {\phi}_p(u^*, u_n )-p{\lambda}_n\left({\phi}_p (u^*, u_n)- \frac{p}{q}{\| u^*\|}^q \right)-p{\lambda}_n \left\langle u_n-u^*,v^*-v_n\right\rangle \nonumber\\
 &&+p{\lambda}_n{\theta}_n {\phi}_p(u^*, u_1 ) +p{\lambda}_nM_1Q_1\nonumber\\
 &=&  \left(1-p{\lambda}_n\right){\phi}_p (u^*, u_n)+p{\lambda}_n\frac{p}{q}{\| u^*\|}^q-p{\lambda}_n \left\langle u_n-u^*,v^*-v_n\right\rangle \nonumber\\
 &&+p{\lambda}_n{\theta}_n {\phi}_p(u^*, u_1 ) +p{\lambda}_nM_1Q_1.
 \end{eqnarray}
 
 Similarly,
  \begin{eqnarray*}
{\phi}_p(v^*, v_{n+1})
 & = & {\phi}_p\left(v^*, J^E\left(J^{E^*}v_n - {\lambda}_n\left(Kv_n+u_n+{\theta}_n(J^{E^*}v_n-J^{E^*}v_1)\right)\right)\right)\\
 & = & V_p\left(v^*, J^{E^*}v_n - {\lambda}_n\left(Kv_n+u_n+{\theta}_n(J^{E^*}v_n-J^{E^*}v_1)\right)\right)\\
 & \leq & V_p(v^*, J^{E^*}v_n )-p{\lambda}_n\left\langle J^E( J^{E^*}v_n - {\lambda}_n\left(Kv_n+u_n+{\theta}_n(J^{E^*}v_n-J^{E^*}v_1)\right))-v^*, y^*\right\rangle \\
 & & (\mbox{by Lemma ~(\ref{l20}) where} \ y^*{ =} Kv_n+u_n+{\theta}_n(J^{E^*}v_n-J^{E^*}v_1) \\
 &= & {\phi}_p(v^*, v_n )-p{\lambda}_n \left\langle v_n-v^*, y^*\right\rangle\\
 &&-p{\lambda}_n\left\langle J^E( J^{E^*}v_n - {\lambda}_n\left(Kv_n+u_n+{\theta}_n(J^{E^*}v_n-J^{E^*}v_1)\right))-v_n, y^*\right\rangle.
 \end{eqnarray*}
  
 By Schwartz inequality and by applying inequality (\ref{e32}), we obtain
\begin{eqnarray*}
{\phi}_p(v^*, v_{n+1})
& \leq & {\phi}_p(v^*, v_n )-p{\lambda}_n \left\langle v_n-v^*,Kv_n+u_n+{\theta}_n(J^{E^*}v_n-J^{E^*}v_1)\right\rangle + p{\lambda}_nM_2Q_2\\
& = &{\phi}_p(v^*, v_n )-p{\lambda}_n \left\langle v_n-v^*, Kv_n-Kv^*\right\rangle\ \mbox{(since $v^*\in N(K)$)}\\
& &-p{\lambda}_n \left\langle v_n-v^*,u_n-u^*\right\rangle-p{\lambda}_n{\theta}_n \left\langle v_n-v^*, J^{E^*}v_n-J^{E^*}v_1\right\rangle +p{\lambda}_nM_2Q_2.
\end{eqnarray*}
By Lemma \ref{l31}, $p\left\langle v_n-v^*, J^{E^*}v_1-J^{E^*}v_n\right\rangle \leq {\phi}_p(v^*, v_1)- {\phi}_p(v^*, v_n ).$ Consequently,\\ $p\left\langle v_n-v^*, J^{E^*}v_1-J^{E^*}v_n\right\rangle \leq {\phi}_p(v^*, v_1).$ Therefore, using $(p, {\eta}_2)$-strongly monotonicity property of $K,$ we have,
\begin{eqnarray}\label{e101}
{\phi}_p(v^*, v_{n+1})
 &\leq&  {\phi}_p(v^*, v_n )-p{\eta}_2{\lambda}_n{\|v_n -v^* \|}^p-p{\lambda}_n \left\langle v_n-v^*,u_n-u^*\right\rangle \nonumber\\
  &&+p{\lambda}_n{\theta}_n \left\langle v_n-v^*, J^{-1}v_1-J^{-1}v_n\right\rangle +p{\lambda}_nM_2Q_2 \nonumber\\
 &\leq&  {\phi}_p(v^*, v_n )-p{\eta}_2{\lambda}_n{\|v_n -v^* \|}^p+p{\lambda}_n \left\langle u_n-u^*, v_n-v^*\right\rangle \nonumber\\
 & &+p{\lambda}_n{\theta}_n {\phi}_p(v^*, v_1) +p{\lambda}_nM_2Q_2 \nonumber \\
 &\leq&  {\phi}_p(v^*, v_n )-p{\lambda}_n\left({\phi}_p (v^*, v_n)- \frac{p}{q}{\| v^*\|}^q \right)+p{\lambda}_n \left\langle u_n-u^*, v_n-v^*\right\rangle \nonumber\\
 &&+p{\lambda}_n{\theta}_n{\phi}_p(v^*, v_1)+p{\lambda}_nM_2Q_2\nonumber\\
 &=&\left(1-p{\lambda}_n\right){\phi}_p (v^*, v_n)+ p{\lambda}_n\frac{p}{q}{\| v^*\|}^q +p{\lambda}_n \left\langle u_n-u^*, v_n-v^*\right\rangle \nonumber\\
 &&+p{\lambda}_n{\theta}_n{\phi}_p(v^*, v_1)+p{\lambda}_nM_2Q_2.
 \end{eqnarray}
 
 Adding (\ref{e10}) and (\ref{e101}) gives 
\begin{eqnarray*}\label{e104}
{\wedge}_p(x^*, x_{n+1})& \leq & \left(1-p{\lambda}_n\right){\wedge}_p(x^*, x_n)+ p{\lambda}_n\frac{p}{q}{\| x^*\|}^q +p{\lambda}_n{\theta}_n{\wedge}_p(x^*, x_1)+p{\lambda}_nMQ\\
&\leq&\left(1-p{\lambda}_n\right)r+ p{\lambda}_n\frac{r}{4}+p{\lambda}_n\frac{r}{2}+p{\lambda}_n\frac{r}{4}\\
&=&\left(1-p{\lambda}_n+p\frac{{\lambda}_n}{4}+p\frac{{\lambda}_n}{2}+p\frac{{\lambda}_n}{4}\right)r=r.
\end{eqnarray*}
Hence, ${\wedge}_p(x^*, x_{n+1}) \leq r.$ By induction, ${\wedge}_p(x^*, x_n) \leq r  ~~ \forall  ~~ n\in \N.$ Thus, from inequality (\ref{e11}), $\left\{x_n\right\}$ is bounded.

\vskip 0.5 truecm

{\bf Part 2:} We now show that $\left\{x_n \right\}$ converges strongly to a solution of $Ax=0.$ Observe that $u^*$ in $E$ is a solution of $u+KFu=0$ if and only if $x^*=(u^*, v^*)$ is a solution of $Ax=0$ in $X$ for $v^*=Fu^* \in E^*,$ since $Ax = \left(Fu - v, Kv + u \right)$ with $x:=\left(u, v\right).$  This implies that 
\begin{eqnarray*}
Fu^*-v^*=0,\\
Kv^*+u^*=0.
\end{eqnarray*}
We recall that $(p, \eta)$-strongly monotone implies monotone and it is given that the range $(J^B_p+tA)$ is all of $X^*$ for all $t>0.$ By Kohsaka and Takahashi \cite{kt}, since $X$ is a reflexive smooth strictly convex space, we obtain for every $t>0$ and $x\in X,$ there exists a unique $x_t\in X$ such that
\begin{equation}
J^Xx = J^Xx_t+tAx_t.
\end{equation} 
Define $J^X_tx:=x_t,$ in other words, define a single-valued mapping $J^X_t : E\rightarrow D(A)$ by $J^X_t=(J^X+tA)^{-1}J^X.$ Such a $J^X_t$ is called the resolvent of $A.$ 
Setting $t:=\frac{1}{{\theta}_n}$ and by the result of Aoyama et al. \cite{Aoyamaetal1} and Reich \cite{Reich1}, for some $\left(u_1,v_1\right):=x_1\in  X,$ there exists in $X$ a unique sequence
\begin{equation}
\left(y_n, z_n\right):=\left(J^X+\frac{1}{{\theta}_n}A\right)^{-1}J^X\left(u_1,v_1\right)
\end{equation} 
with $\left(y_n, z_n\right)\rightarrow \left(u^*,v^*\right):=x^*\in A^{-1}(0),$ where $A\left(y_n, z_n\right) = \left(Fy_n - z_n, Kz_n + y_n \right).$ It can be obtained that
\begin{equation}\label{e24}
{\theta}_n\left(J^Ey_n-J^Eu_1\right)+Fy_n-z_n=0,
\end{equation} 
\begin{equation}\label{e241}
{\theta}_n\left(J^{E^*}z_n-J^{E^*}v_1\right)+Kz_n+y_n=0,
\end{equation} 
where $\left\{y_n\right\}$ and  $\left\{z_n\right\}$ are known to be bounded since they are convergent sequences. Therefore, it is required to show that $u_n\rightarrow y_n$ and $v_n\rightarrow z_n$ as $n\rightarrow \infty.$  Following the same arguments as in part 1, we obtain,
 \begin{equation}\label{e25}
{\phi}_p(y_n, u_{n+1}) \leq  {\phi}_p(y_n, u_n )-p{\lambda}_n \left\langle u_n-y_n, Fu_n-v_n+{\theta}_n(J^Eu_n-J^Eu_1)\right\rangle+p{\lambda}_nM_1Q_1 \ \mbox{and}
\end{equation}
 \begin{equation}\label{e251}
{\phi}_p(z_n, v_{n+1}) \leq  {\phi}_p(z_n, v_n )-p{\lambda}_n \left\langle v_n-z_n, Kv_n+u_n+{\theta}_n(J^{E^*}v_n-J^{E^*}v_1)\right\rangle+p{\lambda}_nM_2Q_2.
\end{equation}
By the $(p, \eta_1)$-strong monotonicity of $F$ and using Lemma \ref{t3} and Eq. (\ref{e24}), we obtain,
\begin{eqnarray*}
&&\left\langle u_n-y_n, Fu_n-v_n+{\theta}_n(J^Eu_n-J^Eu_1)\right\rangle\\
&= & \left\langle u_n-y_n, Fu_n-v_n+{\theta}_n(J^Eu_n-J^Ey_n+J^Ey_n-J^Eu_1)\right\rangle\\
&= & {\theta}_n\left\langle u_n-y_n,J^Eu_n-J^Ey_n\right\rangle+ \left\langle u_n-y_n, Fu_n-v_n+{\theta}_n(J^Ey_n-J^Eu_1)\right\rangle\\
& = &{\theta}_n\left\langle u_n-y_n,J^Eu_n-J^Ey_n\right\rangle+\left\langle u_n-y_n, Fu_n-v_n-(Fy_n-z_n) \right\rangle\\
&\geq  &{\theta}_ng(\|u_n-y_n \|)+\left\langle u_n-y_n, Fu_n-Fy_n\right\rangle+\left\langle u_n-y_n, z_n-v_n \right\rangle\\
&\geq & {\theta}_ng(\|u_n-y_n \|) +\eta_1{\|u_n-y_n \|}^p+\left\langle u_n-y_n, z_n-v_n \right\rangle\ \mbox{(since $F$ is $(p, \eta_1)$-strongly monotone and by Lemma \ref{t3})} \\
&\geq & \frac{1}{p}{\theta}_n {\phi}_p(y_n, u_n )+\left\langle u_n-y_n, z_n-v_n \right\rangle.
\end{eqnarray*}
Therefore, the inequality (\ref{e25}) becomes
\begin{equation}\label{e29}
{\phi}_p(y_n, u_{n+1}) \leq  (1-{\lambda}_n{\theta}_n){\phi}_p(y_n, u_n )-p{\lambda}_n\left\langle u_n-y_n, z_n-v_n \right\rangle +p{\lambda}_nM_1Q_1.
\end{equation}
By the $(p, \eta_2)$-strong monotonicity of $K$ and using Lemma \ref{t3} and Eq. (\ref{e241}), we obtain,
\begin{eqnarray*}
&&\left\langle v_n-z_n, Kv_n+u_n+{\theta}_n(J^{E^*}v_n-J^{E^*}v_1)\right\rangle\\
&= & \left\langle v_n-z_n, Kv_n+u_n+{\theta}_n(J^{E^*}v_n-J^{E^*}z_n+J^{E^*}z_n-J^{E^*}v_1)\right\rangle\\
&= & {\theta}_n\left\langle v_n-z_n,J^{E^*}v_n-J^{E^*}z_n\right\rangle+ \left\langle v_n-z_n, Kv_n+u_n+{\theta}_n(J^{E^*}z_n-J^{E^*}v_1)\right\rangle\\
& = &{\theta}_n\left\langle v_n-z_n,J^{E^*}v_n-J^{E^*}z_n\right\rangle+\left\langle v_n-z_n, Kv_n+u_n-(Kz_n+y_n) \right\rangle\\
&\geq  &{\theta}_ng(\|v_n-z_n \|)+\left\langle v_n-z_n, Kv_n-Kz_n\right\rangle+\left\langle v_n-z_n, u_n-y_n \right\rangle\\
&\geq & {\theta}_ng(\|v_n-z_n \|) +\eta_2{\|v_n-z_n \|}^p+\left\langle v_n-z_n, u_n-y_n \right\rangle\ \mbox{(since $K$ is $(p, \eta_2)$-strongly monotone and by Lemma \ref{t3})} \\
&\geq & \frac{1}{p}{\theta}_n {\phi}_p(z_n, v_n )+\left\langle v_n-z_n, u_n-y_n \right\rangle.
\end{eqnarray*}
Therefore, the inequality (\ref{e251}) becomes
\begin{equation}\label{e291}
{\phi}_p(z_n, v_{n+1}) \leq  (1-{\lambda}_n{\theta}_n){\phi}_p(z_n, v_n )-p{\lambda}_n\left\langle v_n-z_n, u_n-y_n \right\rangle +p{\lambda}_nM_2Q_2.
\end{equation}
Observe that by Lemma \ref{l31}, we have
\begin{eqnarray}\label{e26}
{\phi}_p(y_n, u_n )&\leq &  {\phi}_p(y_{n-1}, u_n )-p\left\langle y_n-u_n, J^Ey_{n-1}-J^Ey_n\right\rangle \nonumber \\ 
 &=& {\phi}_p(y_{n-1}, u_n )+p\left\langle u_n-y_n, J^Ey_{n-1}-J^Ey_n \right\rangle \nonumber \\
& \leq& {\phi}_p(y_{n-1}, u_n )+\|J^Ey_{n-1}-J^Ey_n\|\|u_n-y_n\|,
\end{eqnarray}
and similarly
\begin{eqnarray}\label{e261}
{\phi}_p(z_n, v_n )&\leq &  {\phi}_p(z_{n-1}, v_n )-p\left\langle z_n-v_n, J^{E^*}z_{n-1}-J^{E^*}z_n\right\rangle \nonumber \\ 
 &=& {\phi}_p(z_{n-1}, v_n )+p\left\langle v_n-z_n, J^{E^*}z_{n-1}-J^{E^*}z_n \right\rangle \nonumber \\
& \leq& {\phi}_p(z_{n-1}, v_n )+\|J^{E^*}z_{n-1}-J^{E^*}z_n\|\|v_n-z_n\|.
\end{eqnarray}

Let $R > 0$ such that $\|u_1\| \leq R, \|y_n\| \leq R$ for all $n \in  \N$. Since $\left\{ {\theta}_n\right\}^{\infty}_{n=1}$ is a decreasing sequence, it is known that ${\theta}_{n-1}\geq {\theta}_n.$ Therefore,
\begin{eqnarray*}
\frac{{\theta}_{n-1}-{\theta}_n}{{\theta}_n}= \frac{{\theta}_{n-1}}{{\theta}_n}-1\geq 0.
\end{eqnarray*}
From Eq.(\ref{e24}), one can obtain that
$$J^Ey_{n-1}-J^Ey_n+\frac{1}{{\theta}_n}\left(Fy_{n-1}-z_{n-1}-(Fy_n-z_n)\right)=  \frac{{\theta}_{n-1}-{\theta}_n}{{\theta}_n}\left(J^Eu_1-J^Ey_{n-1}\right).$$
 By taking the duality pairing of each side of this equation with respect to $y_{n-1}-y_n$ and by the strong monotonicity of $A$, we have
$$\left\langle J^Ey_{n-1}-J^Ey_n, y_{n-1}-y_n\right\rangle+\left\langle z_n-z_{n-1}, y_{n-1}-y_n\right\rangle \leq  \frac{{\theta}_{n-1}-{\theta}_n}{{\theta}_n}\|J^Eu_1-J^Ey_{n-1}\|\| y_{n-1}-y_n\|,$$
which gives,
\begin{equation}\label{e27}
 \|J^Ey_{n-1}-J^Ey_n\| \leq \left( \frac{{\theta}_{n-1}}{{\theta}_n}-1\right)\|J^Ey_{n-1}-J^Eu_1\|.
\end{equation}
Similarly, for $R > 0$ such that $\|v_1\| \leq R, \|z_n\| \leq R$ for all $n \in  \N$, we obtain from Eq.(\ref{e241}) that
\begin{equation}\label{e271}
 \|J^{E^*}z_{n-1}-J^{E^*}z_n\| \leq \left( \frac{{\theta}_{n-1}}{{\theta}_n}-1\right)\|J^{E^*}z_{n-1}-J^{E^*}v_1\|.
\end{equation}
Using (\ref{e26}) and (\ref{e27}), the inequality (\ref{e29}) becomes
 \begin{equation}\label{e28}
{\phi}_p(y_n, u_{n+1}) \leq  (1-{\lambda}_n{\theta}_n){\phi}_p(y_{n-1}, u_n)+C_1\left( \frac{{\theta}_{n-1}}{{\theta}_n}-1\right)-p{\lambda}_n\left\langle u_n-y_n, z_n-v_n \right\rangle + p{\lambda}_nM_1Q_1,
\end{equation}
for some constant $C_1 > 0$ and using (\ref{e261}) and (\ref{e271}), the inequality (\ref{e291}) becomes
 \begin{equation}\label{e281}
{\phi}_p(z_n, v_{n+1}) \leq  (1-{\lambda}_n{\theta}_n){\phi}_p(z_{n-1}, v_n)+ C_2\left( \frac{{\theta}_{n-1}}{{\theta}_n}-1\right)-p{\lambda}_n\left\langle v_n-z_n, u_n-y_n \right\rangle + p{\lambda}_nM_2Q_2,
\end{equation}
for some constant $C_2 > 0$. Adding (\ref{e28}) and (\ref{e281}) gives
$$\wedge(w_n, x_{n+1}) \leq (1-{\lambda}_n{\theta}_n)\wedge(w_{n-1}, x_n)+C\left( \frac{{\theta}_{n-1}}{{\theta}_n}-1\right)+ p{\lambda}_nMQ,$$
 where $w_n:=\left(y_n, z_n\right)$ and $C:=C_1+C_2 > 0.$ By Lemma \ref{l11}, $\phi(w_{n-1}, x_n )\rightarrow 0$ as $n\rightarrow \infty$ and using Lemma \ref{l15}, we have that $x_n-w_{n-1}\rightarrow 0$ as $n\rightarrow \infty$ (See e.g, Chidume and Djitte \cite{r5}). Since $w_n\rightarrow x^* \in N(A),$ we obtain that $x_n\rightarrow x^*$ as $n\rightarrow \infty.$ But $x_n=(u_n, v_n)$ and $x^*=(u^*, v^*),$ this implies that $u_n\rightarrow u^*$ which is the solution of the Hammerstein equation.  
 \end{proof}
 
\begin{corollary}
 Let $H$ be a Hilbert space, $p>1$ and $\eta_1, \eta_2 \in(1,\infty).$ Suppose $F: H \rightarrow H$ is a continuous $(p, \eta_1)$-strongly monotone mapping such that the range of $(I+t_1F)$ is all of $H$ and $K: H \rightarrow H$ is a continuous $(p, \eta_2)$-strongly monotone mapping such that the range of $(I+t_2K)$ is all of $H$ for some $t_1,t_2>0.$ Let $\left\{u_n\right\}$ and $\left\{v_n\right\}$ be the sequences in $H$ defined iteratively for arbitrary points $u_1, v_1\in H$ by
\begin{equation}
u_{n+1} = u_n - {\lambda}_n(Fu_n-v_n)-{\lambda}_n{\theta}_n(u_n-u_1), n \in \N,
\end{equation}
\begin{equation}
v_{n+1} = v_n - {\lambda}_n(Kv_n+u_n)-{\lambda}_n{\theta}_n(v_n-v_1), n \in \N,
\end{equation}
where the real sequences $\left\{ {\lambda}_n\right\}^{\infty}_{n=1}\subset (0,1)$ and $\left\{ {\theta}_n\right\}^{\infty}_{n=1}$ in $(0,\frac{1}{2})$ are such that,
\begin{itemize}
	\item[(i)] $\displaystyle \lim_{n\rightarrow \infty}{\theta}_n =0;$ 
	\item[(ii)]$ \displaystyle\sum_{n=1}^{\infty} {\lambda}_n{\theta}_n=\infty$; ${\lambda}_n=o({\theta}_n)$;
	\item[(iii)]$\displaystyle \lim_{n\rightarrow \infty}\left(({\theta}_{n-1}/{\theta}_n)-1\right)/{{\lambda}_n{\theta}_n}=0.$
\end{itemize}		
 Suppose that $0 = u + KFu$ has a solution in $E.$ Then the sequence $\left\{u_n \right\}$ converges strongly to the solution of $u+KFu=0.$
\end{corollary}

\begin{proof}
Take $E:=H$ and $X := H\times H,$ the result follows from Theorem \ref{t5}. This is true since uniformly smooth and uniformly convex spaces are more general than the Hilbert spaces.
\end{proof}
\section{Application}
An illustration is given to show the application of our results in the physical system. The forced oscillations of finite amplitude of a pendulum is shown as a specific example of nonlinear integral equations of Hammerstein type. This is to convince our readers about the application of our results in solving real life problems, arising in physical phenomena.
\begin{example}
 Let us consider the forced oscillations of finite amplitude of a pendulum (see, e.g., Pascali and Sburlan \cite{b11}, Chapter IV, p. $164$). Consider an inhomogeneous differential equation given by
 \begin{equation} \label{1}
    \begin{cases}
       v''(t) + a^2 sin~ v(t) = z(t), & {\rm}\ t\in[0,1],\\
       v(0) = v(1) = 0.
    \end{cases}       
\end{equation}
The amplitude of oscillation $v(t)$ is a solution of the problem, where the driving force $z(t)$ is periodic and odd. The constant $a\neq 0$ depends on the length of the pendulum and on gravity.
\par We begin by computing the Green function for the $2nd$ order equation,
\begin{equation} \label{d6}
v''(t) =0, ~~v(0) = v(1) = 0.
\end{equation}
Let $\Lcal$ denotes the differential operator and $\Mcal$, the manifold which denotes the differential equation together with the associated boundary conditions. For $x\in[0,1]$, we denote by $x^{-}$, the values of $t\in[0,x)$  and by $x^{+}$, the values of $t\in(x,1]$. For $t\neq x$, the Green function is simply a homogeneous solution of the differential equation. However at $t =x$ we expect some singular behavior. The interested readers may read further for the algorithm for constructing the Green function, $G(t,x)$ for $nth$ order equations (see, e.g., Aibinu and Chidume \cite{b39}). For $n=2$, the algorithm for computing $G(t,x)$ is given as below:
\begin{itemize}
	\item [(i)] $\Lcal \left(G(.,x) \right)(t)=0$ for $0\leq t<x$ and for $x<t \leq 1$;
	\item [(ii)] $G(.,x)$ is in $\Mcal$;
	\item [(iii)]$G(.,x)$ is a continuous function;
	\item [(iv)] $\frac{\partial G(t,x)}{\partial t}/_{_{t=x^{+}}}-\frac{\partial G(t,x)}{\partial t}/_{_{t=x^{-}}}=\frac{1}{c_2(x)}$ (where $c_2(x)$ is the coefficient of the second order term).
\end{itemize}

A pair of solutions to the homogeneous equation $v''=0$ are $1$ and $t$. Therefore, the general solution is given by 
$$v(t)=a_1+a_2t,$$
 where $a_1$ and $a_2$ are constants.
From condition (i), we seek the Green function in the form

\begin{equation} \label{d10}
G(t,x) =  \left \{ \begin{array}{cl} A+Bt, & {\rm}\ 0 \leq t \leq x,\\
 C+Dt,& {\rm}\ x < t \leq 1, \end{array} \right.
\end{equation}
where A,B,C and D are functions of the parameter x. Condition (ii) requires that $G(.,x)$ be in $\Mcal$. Therefore, we evaluate $G(0,x)=0$ and $G(1,x)=0$, which give,
\begin{equation} \label{d7}
A=0 \hspace{.2cm} \mbox{and} \hspace{.2cm} C+D=0.
\end{equation}
By condition (iii), $G(.,x)$ is a continuous function.
$$ G(t,x)|_{t\rightarrow x^+}= G(t,x)|_{t\rightarrow x^-}\Rightarrow (C-A)+(D-B)x=0.$$
Since $A=0$ (from (\ref{d7})), we have,
\begin{equation} \label{d8}
C+(D-B)x=0.
\end{equation}
Condition (iv) requires that $G_t|_{t=x^+}-G_t|_{t=x^-}=1$ (since $c_2(x)=1$). Thus,
\begin{equation} \label{d9}
D-B=1
\end{equation}
Solving ((\ref{d7}), (\ref{d8}) and (\ref{d9})) for the three unknowns, gives $B=x-1,~C=-x$ and $D=x$.\\
 By substituting for the values of $A,B,C$ and $D$ in Eq.(\ref{d10}), we obtain
\begin{equation}
G(t,x) =  \left \{ \begin{array}{cl} t(x-1), & {\rm}\ 0 \leq t \leq x,\\
 x(t-1),& {\rm}\ x < t \leq 1. \end{array} \right.
\end{equation}
Equivalently, the Green function for the given boundary value problem is the triangular function given by
\begin{equation}
-G(t,x) =  \left \{ \begin{array}{cl}  t(1-x), & {\rm}\ 0 \leq t \leq x,\\
 x(1-t),& {\rm}\ x < t \leq 1. \end{array} \right.
\end{equation}
Eq.(\ref{1}) is equivalent to the nonlinear integral equation
\begin{equation}\label{2}
v(t) = - \int_{0}^{1}G(t, x)\left[z(x)-a^2 sin ~v(x) \right]dx.
\end{equation}
Take $ \int_{0}^{1}G(t, x)z(x)dx = g(t)$ and $v(t) + g(t) = u(t),$ then Eq.(\ref{2}) can be written as the integral equation 
\begin{equation}\label{3}
u(t) + \int_{0}^{1}G(t, x)f(x, u(x))dx = 0,
\end{equation}
where $f(x, u(x)) = a^2 sin \left[u(x)-g(x)\right]$.
Eq.(\ref{3}) is a homogeneous integral equation of Hammerstein type. 
\end{example}
\section{Numerical Illustration}
Numerical example is given to depict the convergence of the sequences $\left\{u_n\right\}$ and $\left\{v_n\right\}$ which are defined in the main theorem.
Let $E$ be a Hilbert space. Then the duality map becomes the identity map. Consequently, the sequences (\ref{am1}) and (\ref{am1}) reduce to
\begin{equation}\label{am3}
u_{n+1} = u_n - {\lambda}_n\left(Fu_n-v_n+{\theta}_n(u_n-u_1)\right), n \in \N,
\end{equation}
\begin{equation}\label{am4}
v_{n+1} = v_n - {\lambda}_n\left(Kv_n+u_n+{\theta}_n(v_n-v_1)\right), n \in \N,
\end{equation}
from arbitrary $u_1$ and $v_1$ in $ E.$
\par In particular, let $E={\R}^2$ with the usual norm, then $E^*={\R}^2.$ Let 
$F=\begin{pmatrix} 
7 & 9 \\
-9 & 25 
\end{pmatrix},$  
$K=\begin{pmatrix} 
3 & -2 \\
2 & 5 
\end{pmatrix}$
with 
$\check{u}=\begin{pmatrix} 
u_1  \\
u_2  
\end{pmatrix}$
and
 $\check{v}=\begin{pmatrix} 
v_1  \\
v_2  
\end{pmatrix}.$ Then $F\check{u}=\left(7u_1+9u_2, ~-9u_1+25u_2 \right)$ nad $K\check{v}=\left(3v_1-2v_2, ~2v_1+5v_2\right).$ Therefore
 $\langle F\check{u}, \check{u}\rangle \ge 7{\|\check{u}\|}^2$ and $\langle K\check{v}, \check{v}\rangle\ge 3{\|\check{v}\|}^2.$  Thus for $p=2,$ $F$ is a $(p, \eta_1)$-strongly monotone mapping and $K$ is a $(p, \eta_2)$-strongly monotone mapping with $\eta_1=7$ and $\eta_2=3,$ respectively. Take $\left\{ {\lambda}_n\right\}=\left\{ \frac{1}{n} \right\}$ and  $\left\{{\theta}_n\right\}=\left\{\frac{1}{n+1}\right\},$ numerical results are given for the solution of Hammerstein integral equation by Matlab 2015a while the tolerance $10^{-4}$ is being used.

\begin{table}
\begin{center}
\caption{Numerical values of $||u_{n+1}-u_n||$ for n$^{th}$ iterations}
\label{Hammt1}
\begin{tabular}{c c c c}
\hline
\multirow{3}{*}{Iteration} & \multicolumn{3}{c}{$||u_{n+1}-u_n||$} \\
\cline{2-4}
 &  $\left(u_1=(1,1), v_1=(1,1)\right)$ & $\left(u_1=(1,1/2),v_1=(1/4,1)\right)$ & $\left(u_1=(4,-5),v_1=(-7,3)\right)$   \\
\cline{2-4}
  (n) & e+04  & e+03 &  e+03   \\
\hline
  1& 0.0017 & 0.0103 &0.0610       \\ 
  2& 0.0130 & 0.0848 &0.3039\\ 
  3 &0.0610 & 0.4167 &1.0488  \\
  4& 0.1938 & 1.3594 &2.6901 \\
  5& 0.4470 & 3.1870 &5.2733 \\
  6& 0.7767 & 5.6002 &8.0423  \\
  7& 1.0380 & 7.5473 &9.6402 \\ 
  8& 1.0775 & 7.8854 &9.1093 \\
  9& 0.8695 & 6.3964 &6.7588  \\
  10& 0.5410 & 3.9983&3.8921  \\ 
  11&0.2549 & 1.8913 & 1.7018 \\ 
  12&0.0879 & 0.6546 &0.5438 \\ 
  13& 0.0208 & 0.1558 &0.1185\\ 
  14& 0.0030 & 0.0224 &0.0152  \\ 
  15&  0.0002 & 0.0013 & 0.0007\\
  16& 0.0000  & 0.0000 & 0.0000 \\ 
  17& 0.0000 & 0.0000 & 0.0000 \\ 
  18& 0.0000 & 0.0000 & 0.0000 \\ 
  19& 0.0000 & 0.0000 & 0.0000 \\ 
  20& 0.0000 & 0.0000 & 0.0000 \\  
  \hline
\end{tabular}
\end{center}
\end{table}
 
 \begin{figure}
\includegraphics[width=5.0cm ,height=5.0cm]{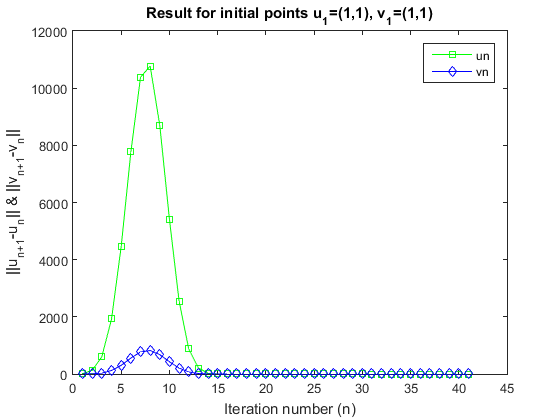}
\includegraphics[width=5.0cm ,height=5.0cm]{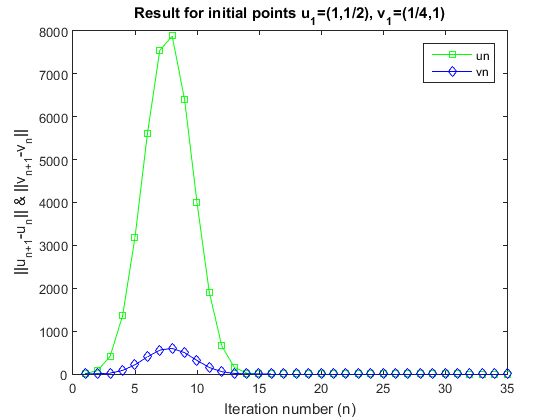}
\includegraphics[width=5.0cm ,height=5.0cm]{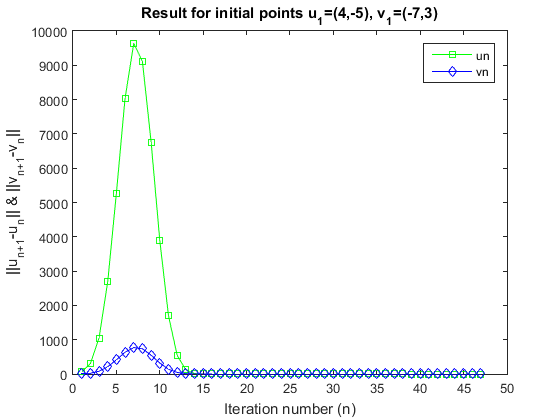}
 \caption{Numerical results}
 \label{Hammf1}
\end{figure}

\section*{Conclusion}
The iterative algorithm for the solution of nonlinear integral equations of Hammerstein type with monotone mappings has been considered. The strong convergence of the sequences of iteration to the solution of nonlinear integral equations of Hammerstein type is obtained without the assumption of existence of a real constant whose calculation is unclear. This shows the efficacy of the technique which has been used in this study in getting rid of the assumption of existence of a real constant whose calculation is unclear and how some results which were obtained in Hilbert space can be extended to a general Banach space (See e.g, Chidume and Djitte \cite{r5}). An illustration is given to convince our readers about the application of our results in solving some real life problems which is common in physical sciences. The forced oscillations of finite amplitude of a pendulum was shown as a specific example of nonlinear integral equations of Hammerstein type. The numerical example portrays the convergence of the sequences $\left\{u_n\right\}$ and $\left\{v_n\right\}.$

\footnotesize
\noindent {\bf Acknowledgements}: The first author acknowledges with thanks the postdoctoral fellowship and financial support from the DSI-NRF Center of Excellence in Mathematical and Statistical Sciences (CoE-MaSS). Opinions expressed and conclusions arrived are those of the authors and are not necessarily to be attributed to the CoE-MaSS.
 


\end{document}